\documentclass[12pt]{article}
\usepackage{amsmath,amssymb,amsbsy,amsfonts,amsthm,latexsym,amsopn,amstext,
                                amsxtra,euscript,amscd}

\begin{document}

\newtheorem{theorem}{Theorem}
\newtheorem{lemma}[theorem]{Lemma}
\newtheorem{claim}[theorem]{Claim} 
\newtheorem{cor}[theorem]{Corollary}
\newtheorem{prop}[theorem]{Proposition}
\newtheorem{definition}{Definition}
\newtheorem{question}[theorem]{Open Question}
\def\bE{{\mathbf E}}
\def\mand{\qquad\mbox{and}\qquad}
\def\scr{\scriptstyle}
\def\\{\cr}
\def\({\left(}
\def\){\right)}
\def\[{\left[}
\def\]{\right]}
\def\<{\langle}
\def\>{\rangle}
\def\fl#1{\left\lfloor#1\right\rfloor}
\def\rf#1{\left\lceil#1\right\rceil}
\def\cA{{\mathcal A}}
\def\cB{{\mathcal B}}
\def\cC{{\mathcal C}}
\def\cD{{\mathcal D}}
\def\cE{{\mathcal E}}
\def\cF{{\mathcal F}}
\def\cG{{\mathcal G}}
\def\cH{{\mathcal H}}
\def\cI{{\mathcal I}}
\def\cJ{{\mathcal J}}
\def\cK{{\mathcal K}}
\def\cL{{\mathcal L}}
\def\cM{{\mathcal M}}
\def\cN{{\mathcal N}}
\def\cO{{\mathcal O}}
\def\cP{{\mathcal P}}
\def\cQ{{\mathcal Q}}
\def\cR{{\mathcal R}}
\def\cS{{\mathcal S}}
\def\cT{{\mathcal T}}
\def\cU{{\mathcal U}}
\def\cV{{\mathcal V}}
\def\cW{{\mathcal W}}
\def\cX{{\mathcal X}}
\def\cY{{\mathcal Y}}
\def\cZ{{\mathcal Z}}
\def\N{{\mathbb N}}
\def\Z{{\mathbb Z}}

\def\eps{\varepsilon}
\def\mand{\qquad\text{and}\qquad}
\def\tS{S^*}
\def\tsigma{\widetilde \sigma}
\def\C{{\mathbb C}}
\def\F{{\mathbb F}}
\def\Fp{\F_p}
\def\Fq{\F_q}
\def\E{{\mathcal E}}
\def\e{{\bf e}}
\def\ep{\e_p}
\def\O{{\cO}}
\def\x{{\bf x}}
\def\y{{\bf y}}

\renewcommand{\vec}[1]{\mathbf{#1}}
\def\Or{{\mathcal O}}

\newcommand{\comm}[1]{\marginpar{%
\vskip-\baselineskip 
\raggedright\footnotesize
\itshape\hrule\smallskip#1\par\smallskip\hrule}}



\title{Exponential Sums over Points of 
Elliptic Curves}

\author{
 {\sc Omran Ahmadi}  \\
{School of Mathematics}\\{ Institute for Research in Fundamental Sciences}\\
{P.O. Box: 19395-5746, Tehran, Iran}\\
{\tt oahmadid@gmail.com}
\and
{\sc Igor E.~Shparlinski} \\
{Department of Pure Mathematics}\\
{ University of New South Wales, } \\
{Sydney, NSW 2052, Australia} \\
{\tt igor.shparlinski@unsw.edu.au}
}

\date{}

\maketitle

\begin{abstract}
We derive a new bound for some bilinear sums over points 
of an elliptic curve over a finite field. We use this bound
to improve a series of previous results on various 
exponential sums and some arithmetic problems 
involving   points on elliptic curves. 
\end{abstract}

\paragraph{Subject Classification (2010)}
\ Primary 11L07, 11T23 \ Secondary 11G20

\section{Introduction}

Let $q$ be a prime power and let $\E$ be an elliptic curve defined over 
a finite field $\F_q$ of $q$ elements of characteristic $p\ge 5$
given by an affine Weierstra\ss\ equation
$$\E:\quad Y^2=X^3+ AX+ B$$
with some $A,B \in \F_q$, 
see~\cite{ACDFLNV,BSS,Silv}.

We recall that the set of all points on $\E$ forms an abelian group, 
with the ``point at infinity'' $\cO$ as the neutral 
element,  and
we use $\oplus$ to denote the group operation. In particular, we 
sometimes work with group characters associated with this 
group.

As usual, we write every point $P \ne \cO$ on $\E$ 
as $P = (\x(P), \y(P))$. Let $\E(\F_q)$  denote the set of $\F_q$-rational
points on  $\E$. We recall that the celebrated result of 
Bombieri~\cite{Bomb} implies, in particular, an estimate of 
order $q^{1/2}$ for  exponential  sums with functions
from the function field of $\E$ taken over all points of 
$\E(\F_q)$. More recently, various character sums over points of 
elliptic curves have been considered in a 
number of papers, 
see~\cite{AhmShp,BFGS,Chen,ElMaSh,FarShp,KohShp,LanShp1,LanShp2,OstShp1,OstShp2,Shp1,Shp3,ShpVol}
and references therein. These estimates are motivated by various applications to such areas as
\begin{itemize}
\item pseudorandom number generators from
elliptic curves, see the most recent works~\cite{BOS,Chen,Liu,Mer0,Mer1,Mer2} 
and also the survey~\cite{Shp4};
\item randomness extractors from elliptic curves~\cite{CFPZ,CiSo};
\item analysing an attack on the Digital Signature Algorithm 
on elliptic curves~\cite{NgShp};
\item hashing to elliptic curves~\cite{FFSTV};
\item finding generators and the structure of the groups of points on 
elliptic curves~\cite{KohShp,ShpVol};
\item constructing some special bases related to quantum computing~\cite{ShpWint}.
\end{itemize}

We fix a nonprincipal additive character $\psi$ of $\F_q$.
All of our estimates are uniform with respect to the additive character  $\psi$.  

Let $G\in\E(\F_q)$ be a point of order $T$, in other words, $T$ is
the cardinality of the cyclic group $\langle G\rangle$ generated
by $G$ in $\E(\F_q)$. 

Given two sets  $\cA,\cB \subseteq \Z_T^*$, in the unit group of residue ring 
$\Z_T$ modulo $T$,  and arbitrary  
complex functions $\alpha$ and $\beta$ 
supported on $\cA$ and  $\cB$ with
$$
|\alpha_a| \le 1, \ a \in \cA, \mand  |\beta_b| \le 1, \ b \in \cB, 
$$
we consider the bilinear sums  of {\it multiplicative type\/}:
\begin{equation}
\label{eq:Sum U}
U_{\alpha, \beta} (\psi,\cA, \cB;G) = \sum_{a \in \cA} \sum_{b \in \cB} 
\alpha_a \beta_b\psi(x(abG)).
\end{equation} 

Furthermore, given two sets  $\cP,\cQ \subseteq \E(\F_{q})$ and arbitrary  
complex functions $\rho(P)$ and $\vartheta(Q)$ 
supported on $\cP$ and  $\cQ$ 
we consider the bilinear sums of {\it additive type\/}:
\begin{equation}
\label{eq:Sum V}
V_{\rho, \vartheta} (\psi,\cP, \cQ) = \sum_{P \in \cP} \sum_{Q \in \cQ} 
\rho(P)\vartheta(Q) \psi(x(P\oplus Q)).
\end{equation} 

Bounds of the  sums  
$U_{\alpha, \beta} (\psi,\cA, \cB;G)$ and $V_{\rho, \vartheta} (\psi,\cP, \cQ)$ are proved in~\cite{AhmShp, BFGS}
and~\cite{Shp1}, respectively, where several applications of these
bounds have been  shown.

Here we improve the bound of~\cite{Shp1} and use it with the bound 
of~\cite{AhmShp}, and also with some additional arguments, 
to refine a series of previous results.
In particular, we give improvements:
\begin{itemize}
\item of the elliptic curve version of the sum-product
theorem of~\cite{Shp2};

\item of the bound of character sums from~\cite{LanShp2} with
 sequences of points of cryptographic significance;

\item of the bound of  character  sums from~\cite{Shp3} with
linear combinations of $x(P)$ and $x(nP)$ for $P\in \E(\F_q)$.

\end{itemize}

Throughout the paper, any implied constants in the symbols $O$ and
$\ll$ may occasionally depend, where obvious, on the 
integer parameter $\nu\ge 1$ and real parameter $\varepsilon >0$, but are absolute otherwise. We recall that the
notations $A \ll B$ and  $A = O(B)$ are both equivalent to the
statement that the inequality $|A| \le c\,B$ holds with some
constant $c> 0$.

\section{Preparations}

\subsection{Single sums}

We recall the following special case of the bound of~\cite[Corollary~1]{KohShp}:
 
\begin{lemma}
\label{lem:Subgr}
Let $\E$ be an ordinary curve defined over $\Fq$ and let $G \in \E(\Fq)$ be a point of
order $T$. Then for any group character $\chi$ on $\E(\F_q)$. 
$$
   \sum_{n \in \Z_T}
\psi\( x\(n G\)\) \chi(G)  \ll q^{1/2}.
$$
\end{lemma}

\subsection{Bilinear sums of multiplicative type}

We recall the bound of~\cite[Theorem~2.1]{AhmShp} on the sums~\eqref{eq:Sum U}:

\begin{lemma}
\label{lem:BillSum-M}
Let $\E$ be an ordinary elliptic
curve defined over $\F_q$, and let $G \in \E(\F_q)$ be a point of
order $T$. Then,  for any fixed integer $\nu\ge 1$, uniformly over all 
nontrivial additive characters $\psi$
of $\F_q$, we have
\begin{equation*}
\begin{split}
U_{\alpha, \beta} (\psi,\cA, & \cB;G) \\
\ll & (\#\cA)^{1-1/2\nu} (\#\cB)^{1-1/(\nu+2)}T^{(\nu+1)/\nu(\nu+2)}
q^{1/4(\nu+2)}(\log q)^{1/(\nu+2)}.
\end{split}
\end{equation*}
\end{lemma}

\subsection{Bilinear sums of additive type}

For the sum~\eqref{eq:Sum V}
it is shown in~\cite{Shp1} that if 
$$
\max_{P \in \cP} |\rho(P)|   \le 1
\mand \max_{Q \in \cQ} 
|\vartheta(Q)|\le 1
$$
then for any fixed integer $\nu\ge 1$  we have
\begin{equation}
\label{eq:Bound V}
V_{\rho, \vartheta} (\psi,\cP, \cQ)  \ll (\# \cP)^{1-1/2\nu} 
(\# \cQ)^{1/2} q^{1/2\nu}  +(\#
\cP)^{1-1/2\nu} \# \cQ q^{1/4\nu}.
\end{equation} 
Here we obtain a different bound which is stronger
than~\eqref{eq:Bound V} in several cases (for example,
when $\#\cP = \# \cQ$). 

\begin{theorem}
\label{thm:BillSum-A}
Let $\E$ be an ordinary elliptic
curve defined over $\F_q$ and let
$$
\sum_{P \in \cP} |\rho(P)|^2  \le R
\mand \sum_{Q \in \cQ} 
|\vartheta(Q)|^2  \le T.
$$
 Then,  uniformly over all 
nontrivial additive characters $\psi$of $\F_q$, 
$$
|V_{\rho, \vartheta} (\psi,\cP, \cQ)| \ll \sqrt{q RT}. 
$$
\end{theorem}

\begin{proof} Let $\cX$ be the set of group characters on $\E(\F_q)$. 
We collect the points $P$ and $Q$ with a given sum $S = P \oplus Q$
and identify this condition via the character sum over $\cX$. 
This gives 
$$
V_{\rho, \vartheta} (\psi,\cP, \cQ)  = \sum_{S \in \E(\F_q)}
\psi(x(S)) \sum_{P \in \cP} \sum_{Q \in \cQ} 
\rho(P)\vartheta(Q) \frac{1}{\#\E(\F_q)} \sum_{\chi\in \cX} \chi(P \oplus Q \ominus S).
$$ 
Therefore
\begin{equation*}
\begin{split}
V_{\rho, \vartheta} (\psi,\cP, \cQ)  = 
\frac{1}{\#\E(\F_q)} \sum_{\chi\in \cX} &  \sum_{S \in \E(\F_q)}
\psi(x(S)) \overline{\chi(S)}\\
&\sum_{P \in \cP} \rho(P) \chi(P) \sum_{Q \in \cQ} 
\vartheta(Q) \chi(Q).
\end{split}
\end{equation*} 
The sums over $S$ is $O(q^{1/2})$ by Lemma~\ref{lem:Subgr}, so 
 $$
V_{\rho, \vartheta} (\psi,\cP, \cQ)  \ll  
\frac{q^{1/2}}{\#\E(\F_q)} \sum_{\chi\in \cX}   
\left|\sum_{P \in \cP} \rho(P) \chi(P)\right|
\left| \sum_{Q \in \cQ} \vartheta(Q) \chi(Q)\right|.
$$ 
We now use the Cauchy inequality, getting 
\begin{eqnarray*}
\lefteqn{\(\sum_{\chi\in \cX}   
\left|\sum_{P \in \cP} \rho(P) \chi(P)\right|
\left| \sum_{Q \in \cQ} \vartheta(Q) \chi(Q)\right|\)^2}\\
& & \qquad \le \sum_{\chi\in \cX}   
\left|\sum_{P \in \cP} \rho(P) \chi(P)\right|^2 \sum_{\chi\in \cX}  
\left| \sum_{Q \in \cQ} \vartheta(Q) \chi(Q)\right|^2\\
& & \qquad \le \#\E(\F_q)^2 RT,
\end{eqnarray*}
since
\begin{eqnarray*}
\sum_{\chi\in \cX}  
\left|\sum_{P \in \cP} \rho(P) \chi(P)\right|^2
& =& \sum_{P_1,P_2 \in \cP}\rho(P_1) \overline{\rho(P_2)}
\sum_{\chi\in \cX}  \chi(P_1 \ominus P_2)\\
&=&  \#\E(\F_q) \sum_{P  \in \cP} |\rho(P)|^2
\le \#\E(\F_q)  R
\end{eqnarray*}
Similarly, 
$$
\sum_{\chi\in \cX}  
\left| \sum_{Q \in \cQ} \vartheta(Q) \chi(Q)\right|^2
\le \#\E(\F_q)  T, 
$$
and the desired result now follows. 
\end{proof}

\section{Combinatorial Problems}

\subsection{Sum-product problem for elliptic curves}

In~\cite{Shp2}, for any sets $\cR, \cS \subseteq \E$
it is shown that
\begin{equation}
\label{EC:Sum-Prod-08}
\# \cU \#\cV \gg 
\min\{ q\# \cR,  (\# \cR)^{2} \#\cS q^{-1/2}\}, 
\end{equation}
where 
\begin{equation}
\label{eq:U and V}
\begin{split}
\cU & = \{x(R) + x(S)~:~R \in \cR,\  S \in \cS\},\\
\cV & = \{x(R\oplus S)~:~R \in \cR,  \  S \in \cS\}.
\end{split} 
\end{equation}
Clearly~\eqref{EC:Sum-Prod-08}  implies that at least one of the sets $\cU$
and $\cV$ is large. 

The main ingredient of the proof of~\eqref{EC:Sum-Prod-08}
in~\cite{Shp2} is~\eqref{eq:Bound V}. Using Theorem~\ref{thm:BillSum-A}
in the argument of~\cite{Shp2} one 
immediately derives the following improvement on~\eqref{EC:Sum-Prod-08}: 

\begin{theorem}
Let $\E$ be an ordinary elliptic curve defined over $\Fq$ and let
$\cR$ and  $\cS$ be arbitrary subsets of $\E(\F_{q})$.  
Then for the sets $\cU$ and $\cV$, given by~\eqref{eq:U and V},  we have
$$
\# \cU \#\cV \gg 
\min\{ q\# \cR, (\# \cR\#\cS)^{2}q^{-1}\}.
$$
\end{theorem}

\subsection{S{\'a}rk{\"o}zy problem for elliptic curves}

In~\cite{Shp1}, the number of solutions
$M(\cS, \cT, \cU, \cV)$ of the equation
$$
x(S)+x(T) = x(U \oplus V), \qquad S \in \cS,\ T\in  \cT, \ U \in \cU, 
\ V \in \cV,
$$
for any sets $\cS, \cT, \cU, \cV \subseteq \E(\F_q)$ is estimated.  
It is shown that if
$$
\#\cS \#\cT \#\cU \#\cV \ge q^{7/2+\varepsilon}, \qquad \varepsilon > 0, 
$$
then
\begin{equation}
\label{eq:Sarkozy}
  M(\cS, \cT, \cU, \cV) = \(1 + O(q^{-\varepsilon/2})\) \frac{\#\cS 
\#\cT \#\cU \#\cV}{q} .
\end{equation}

The result above is the elliptic curve analogue of a result
of A.~S{\'a}rk{\"o}zy~\cite{Sark} regarding  the number of solutions
$N(\cA, \cB, \cC, \cD)$ of  the equation
$$
a + b = cd, \qquad a \in \cA, \ b\in  \cB,\ c \in \cC,\ d \in \cD,
$$
for sets
$\cA, \cB, \cC, \cD \subseteq \F_q$. 

In~\cite{Shp1}, the asymptotic formula~\eqref{eq:Sarkozy} 
is proved using~\eqref{eq:Bound V}. 
Now, using Theorem~\ref{thm:BillSum-A}, the following 
improvement on~\eqref{eq:Sarkozy} is immediate. The proof is omitted as it
is completely similar to the proof given in~\cite{Shp1}.

\begin{theorem} Let $\E$ be an ordinary elliptic curve defined over $\Fq$.
Then for every $\varepsilon >0$ and arbitrary sets $\cS, \cT, \cU, \cV \subseteq \E(\F_q)$ with
$$
\#\cS \#\cT \#\cU \#\cV \ge q^{3+\varepsilon}, \qquad \varepsilon > 0
$$
we have
$$
M(\cS, \cT, \cU, \cV) = \(1 + O(q^{-\varepsilon/2})\) \frac{\#\cS 
\#\cT \#\cU \#\cV}{q}\cdot
$$ 
\end{theorem}

\subsection{Distribution of subset sums}

Let $P \in \E(\F_q)$ be an $\F_q$-rational point on an elliptic curve $\E$ over $\F_q$, and $\sigma$ be an endomorphism on $\E$. 
Also, let  $\cM_k$ be the set of $k$-dimensional vectors with
coordinates $0, \pm 1$ which do not have two consecutive nonzero
components, that is,
\begin{equation}
\label{non-aj}
      \mu_j\mu_{j+1} = 0, \mbox{ for all }j = 0, \ldots, k-2.
\end{equation}
Motivated by applications to pseudo-random number generation,  the
set of points
\begin{equation}
\label{eq:Points_Gen}
P_{\sigma,\vec{m}} = \sum_{j=0}^{k-1} \mu_j \sigma^j(P), \qquad
\vec{m} = (\mu_0,
\ldots, \mu_{k-1}) \in \cM_k,
\end{equation}
where $\sigma$ is an endomorphism of the elliptic curve $\E$ have been  
considered in~\cite{LanShp2}.

In~\cite{LanShp2}, three specific endomorphisms are considered. 
The first endomorphism considered in~\cite{LanShp2} is the {\em doubling} endomorphism $\delta(P) = 2P$
which is defined for any elliptic curve over any
finite field. 

The second endomorphism considered in~\cite{LanShp2}
is the {\em Frobenius endomorphism} of the so called {\em Koblitz curves}.
A Koblitz curve, $\E_a$, $a \in \F_2$, is given by the
Weierstra\ss\ equation
$$
      \E_a: \ Y^2 + XY =X^3 + a X^2 + 1,
$$
(see~\cite{Kob}) and its Frobenius
endomorphism $\varphi$,  which acts on a $\F_{2^n}$-rational point $P
= (x,y) \in \E_a(\F_{2^n})$
is given by
$$
\varphi(P) = (x^2, y^2).
$$
Clearly $\varphi(P) \in \E_a(\F_{2^n})$. 

Finally, as in~\cite{LanShp2}, we consider  one of the so-called GLV curves introduced by
Gallant, Lambert and Vanstone~\cite{GLV}, which we detail below. 

Let the characteristic of $\F_q$ be a prime $p\ge 3$ such that $-7$ is a quadratic residue modulo
$p$ (that is, $p \equiv 1,2,4 \pmod 7$).  Define an elliptic curve 
$\E_{GLV}$ over $\F_p$ by
$$
\E_{GLV}: Y^2=X^3- \frac{3}{4} X^2 -2X -1.
$$
Let $\xi \in \F_p$ be a square root of $-7$.
If $b=(1+\xi)/2$ and $c=(b-3)/4$, then the map $\psi$,
defined in the affine plane by
$$\psi(P)=\(\frac{x^2-b}{b^2(x-c)},\frac{y(x^2-2cx+b)}
{b^3(x-c)^2}\)$$
for $P= (x,y) \in \E_{GLV}$, 
is an endomorphism of $\E_{GLV}$. 

In~\cite{LanShp2}, it has been
shown that under mild conditions,
the points~\eqref{eq:Points_Gen} possess some  uniformity of distribution
properties, where $\sigma$ is one of the following endomorphisms:
\begin{equation}
\label{eq:endom}
\sigma = \left\{ \begin{array}{ll}
 \delta, & \text{for an  arbitrary curve $\E$},\\
 \varphi, & \text{for a Koblitz curve $\E = \E_a$, $a=0,1$}\\
  \psi, & \text{  for the GLV curve $\E = \E_{GLV}$}.
  \end{array} \right.
\end{equation}

Here, using Theorem~\ref{thm:BillSum-A} we improve the result of~\cite{LanShp2}
in some ranges of parameters. 

First we need the following estimate on $\#\cM_k$ given by Bosma~\cite[Proposition~4]{Bosma}.

\begin{lemma}
For any $k\ge 2$, we have:
$$
\#\cM_k=\frac{4}{3}2^k+O(1).
$$
\end{lemma}

For  an endomorphism $\sigma$ of an elliptic curve $\E$ over $\F_q$
and a nonprincipal additive character $\psi$ of $\F_q$, we define the exponential sum
$$
S_{\sigma,k}(\chi) = \sum_{\vec{m} \in\cM_k}
\chi\(x(P_{\sigma,\vec{m}})\),
$$
where we always assume that the value of the character
is defined as zero if the expression in the argument is not defined
(for example,
if  $P_{\sigma,\vec{m}} = \Or$ in the above sum).

It is shown in~\cite[Lemma~2.1]{LanShp2} that if  $P \in \E(\F_q)$ is of prime order $\ell$
then for any  integer $k\ge 1$  the bound
\begin{equation}
\label{eq:S sigma}
| S_{\sigma,k}(\chi) | \ll \# \cM_k
\(q^{1/4\nu} \ell^{-1/2\nu} + 2^{ - k/2\nu}q^{(\nu+1)/4\nu^2}\)
\end{equation}
holds with any fixed integer
$$
\nu \ge \frac{\log q}{2k \log 2},
$$
   where $\sigma$ is one of the endomorphisms~\eqref{eq:endom}.

Given an endomorphism $\sigma$ of an elliptic curve $\E$ over $\F_q$,
and an integer $k \ge 1$,
    we denote by $N_{\sigma,k}(Q)$
the number of representations
$$
      P_{\sigma,\vec{m}}  = Q, \qquad  \vec{m} =  (m_0, \ldots, m_{k-1}) \in
      \cM_k.
$$
We  recall~\cite[Lemma~2.1]{LanShp2}:

\begin{lemma}\label{lem:coll}
Let $\E$ be an ordinary elliptic curve defined over $\Fq$ and let
$P \in \E(\F_q)$ be of prime order $\ell$. Then for any positive integer $k$ and for every point $Q\in \E(\F_{q})$ the
      bound
$$
N_{\sigma,k}(Q) \ll 2^k \ell^{-1} + 1
$$
holds, where $\sigma$ is one of the endomorphisms~\eqref{eq:endom}.
\end{lemma}

We now obtain a bound that improves~\eqref{eq:S sigma} for some values
of parameters (namely for large $k$ and $\ell$).

\begin{theorem}\label{thm:Exp_Sum}
Let $\E$ be an ordinary elliptic curve defined over $\Fq$ and 
let  $P \in \E(\F_q)$ be of prime order $\ell$.
Then for any  integer $k\ge 1$  the bound
$$
| S_{\sigma,k}(\chi) | \ll \# \cM_k q^{1/2}  \ell^{-1} + (\# \cM_k)^{1/2}  q^{1/2}
$$
holds where $\sigma$ is one of the endomorphisms~\eqref{eq:endom}.
\end{theorem}

\begin{proof} Let us choose $r = \rf{k/2}$.
  For $j=0, 1$ we define $\cU_j$ to be the subset of $\vec{u} = (u_1,
    \ldots, u_r) \in \cM_r$ with $u_r = \pm j$.  To form a vector in
    $\cM_k$, a vector from $\cU_0$ can be appended by any vector from 
    $\cV_0 = \cM_{k-r}$, while a vector from $\cU_1$ requires the following
    digit to be zero. Hence, we put
    $$
    \cV_1 = \left\{ (0, \vec{w})~:~\vec{w}\in \cM_{k-r-1}\right\}.
    $$
    We have
    $$
    S_{\sigma,k}(\chi) = R_{\sigma,0} + R_{\sigma,1},
    $$
    where
    $$
    R_{\sigma,j} = \sum_{\vec{u} \in \cU_j} \sum_{\vec{v} \in \cV_j}
    \chi\(x\(P_{\sigma,\vec{u}} + \sigma^r(P_{\sigma,\vec{v}})\)\),
    \qquad j =0,1.
    $$
We now consider the sets 
$$
\cX_j = \{P_{\sigma,\vec{u}}~:~\vec{u} \in \cU_j\}
\mand 
\cY_j = \{\sigma^r(P_{\sigma,\vec{v}})~:~\vec{v} \in \cV_j\}.
$$
Using Lemma~\ref{lem:coll}, we see that we can write
$$
R_{\sigma,j}= \sum_{X \in \cX_j} \sum_{Y \in \cY_j}M(X) N(Y)
    \chi\(x\(S + T\)\),
    \qquad j =0,1, 
    $$
with some positive coefficients $M(X)$ and  $N(Y)$ such that
$$
M(X) \ll 2^r \ell^{-1} + 1 \mand N(Y) \ll 2^{k-r} \ell^{-1} + 1.
$$
We also trivially have 
$$
\sum_{X \in \cX_j} M(X) = \# \cU_j \mand
 \sum_{Y \in \cY_j} N(Y) = \#\cV_j.
$$
Therefore 
$$
\sum_{X \in \cX_j} M(X)^2 \le  \# \cU_j \(2^r \ell^{-1} + 1\)\quad \text{and}\quad
 \sum_{Y \in \cY_j} N(Y)^2 = \#\cV_j  \(2^{k-r} \ell^{-1} + 1\).
$$
Therefore, by  Theorem~\ref{thm:BillSum-A}, we derive
$$
R_{\sigma,0}\ll
\sqrt{q  (2^r \ell^{-1} + 1)(2^{k-r} \ell^{-1} + 1)  \#\cU_j \#\cV_j},  \qquad j =0,1.
$$
Clearly  $\#\cU_j \#\cV_j\le \# \cM_k\ll 2^k$. Furthermore, 
by the choice of $r$
$$
(2^r \ell^{-1} + 1)(2^{k-r} \ell^{-1} + 1) \ll 
(2^{k/2} \ell^{-1} + 1)^ 2 \ll 2^{k} \ell^{-2} + 1.
$$ 
 And  thus
 $$
 |R_{\sigma,j} |\ll q^{1/2}  2^{k} \ell^{-1} + q^{1/2}  2^{k/2}, \qquad j =0,1, 
 $$  
which concludes the proof. 
\end{proof}

Clearly, if for some fixed $\varepsilon > 0$ we have $\ell > q^{1/2+
    \varepsilon}$ and $2^k \ge q^{1+ \varepsilon}$, then
the bound of Theorem~\ref{thm:Exp_Sum} is nontrivial.
As in~\cite[Section~4]{LanShp2}, we can now use this bound 
in various questions about the distribution of $x(P_{\sigma,\vec{m}})$ for $\vec{m} \in\cM_k$.

\section{Sums Over Consecutive Intervals}

\subsection{Stationary phase sums}

For an integer $n$ and $a,b \in \F_q$,  we now consider the sums
$$
S_{n}(\psi; a,b) = \sum_{P \in \E(\F_q)} 
 \psi\(ax(P) +b x(nP)\). 
$$
As it has been mentioned in~\cite{Shp3}, it follows from a much more general result 
of~\cite[Corollary~5]{LanShp1} 
that if at least one of $a$ and $b$ is a non-zero element
of $\F_q$ and $n > 0$, then
\begin{equation}
\label{eq:Gen Bound}
S_{n}(\psi; a,b)  = O(n^2 q^{1/2}).
\end{equation}
Furthermore, in~\cite{Shp3}, the following two bounds are given: 
\begin{equation}\label{eq:Large d}
S_{n}(\psi;a,b) \ll q^{3/2}/d,
\end{equation}
and
\begin{equation}\label{eq:Small d}
 S_{n}(\psi;a,b) \ll qd^{-1/2} + q^{3/4},
\end{equation}
where $d=\gcd\left(n, \#\E(\F_q)\right)$. The above bounds improve on
~\eqref{eq:Gen Bound} when $d$ is not very small. The   bound~\eqref{eq:Large d} is nontrivial 
whenever $d/q^{1/2}\rightarrow\infty$ as $q\rightarrow \infty$. The  bound~\eqref{eq:Small d}
is nontrivial for $d\rightarrow \infty$ as $q\rightarrow \infty$, however it 
is weaker than the first bound for $d>q^{3/4}$.

In~\cite{Shp3}, the bound~\eqref{eq:Bound V} is used to 
obtain~\eqref{eq:Small d}. Here
we use  Theorem~\ref{thm:BillSum-A}, to improve on the 
bounds~\eqref{eq:Large d} and~\eqref{eq:Small d}. Although the proof of the new bound 
is quite similar to the 
proof given in~\cite{Shp3}, here, for the sake of completeness, 
instead of referring for details to ~\cite{Shp3} 
we give a complete proof of this bound.

\begin{theorem}
\label{thm:Sum 2}
Let $\E$ be an ordinary elliptic curve defined over $\Fq$ and let
$n> 0$ be an arbitrary integer.
Then for any $a\in \F_q^*$ and $b \in \F_q$, 
we have
$$
 S_{n}(\psi;a,b) \ll qd^{-1/2}, 
$$
where   $d = \gcd\left(n, \#\E(\F_q)\right)$.
\end{theorem}

\begin{proof} 
Let $\cH_d \subseteq  \E(\F_{q})$ be the 
subgroup of $ \E(\F_{q})$ consisting of the $d$-torsion 
points $Q\in \E(\F_{q})$,  that is, of points $Q$
with $dQ = \cO$. 

It is well-known, see~\cite{ACDFLNV,BSS,Silv}, that  the group 
$\E(\F_{q})$ is isomorphic to
\begin{equation}
\label{eq:Group Struct}
\E(\F_q) \cong \Z_M \times \Z_L
\end{equation}
for some unique integers $M$ and $L$ with
\begin{equation}
\label{eq:Divisibility}
     L \mid M, \qquad LM = \E(\F_{q}),\qquad L \mid q-1.
\end{equation}

Since $d \mid \#\E(\F_q)$ we see from~\eqref{eq:Group Struct}
and~\eqref{eq:Divisibility} that 
we can write $d = d_1 d_2$ where $d_1 =\gcd(d,M)$
and $d_2 \mid d_1$. It is now easy to see 
that 
\begin{equation}
\label{eq:Hd}
 \# \cH_d \ge d, 
\end{equation}
(clearly $\cH_d$ is a subgroup of the group $\E[d]$ of $d$-torsion 
points on $\E$, thus we also have $ \# \cH_d \le d^2$, 
see~\cite{ACDFLNV,BSS,Silv}).

For any point $Q\in \E(\F_{q})$ we have
$$
S_{n}(\psi; a,b)   =  \sum_{P \in \E(\F_q)} 
 \psi\(ax(P\oplus Q) +b x(n(P\oplus Q))\).
$$
Therefore, we obtain
\begin{eqnarray*}
S_{n}(\psi; a,b)  & = & \frac{1}{\# \cH_d}\sum_{Q \in \cH_d} \sum_{P \in \E(\F_q)} 
\psi\(ax(P\oplus Q) +b x(n(P\oplus Q))\)\\
& = &  \frac{1}{\# \cH_d}\sum_{Q \in \cH_d} \sum_{P \in \E(\F_q)} 
 \psi\(ax(P\oplus Q)+bx(nP)\)\\
 & = &  \frac{1}{\# \cH_d} \sum_{P \in \E(\F_q)}\sum_{Q \in \cH_d}  \psi\(bx(nP)\)
\psi\(a x(P\oplus Q)\).
\end{eqnarray*}
Now applying Theorem~\ref{thm:BillSum-A} with $\cP = \E(\F_q)$ and $\cQ =\cH_d$, we have
$$
|S_{n}(\psi; a,b) | \ll  \frac{1}{\# \cH_d} \(q^2 \# \cH_d\)^{1/2}\ll \frac{q}{d^{1/2}}, 
$$
which concludes the proof. 
\end{proof}

Note that for $d\le q$, Theorem~\ref{thm:Sum 2} is an improvement 
on~\eqref{eq:Large d} and ~\eqref{eq:Small d}. If $d> q$, then from
the fact that $\#\E(\F_q)\le q+1+2\sqrt{q}$, see~\cite[Chapter~5, Theorem~1.1]{Silv}, 
it follows that
$d=\#\E(\F_q)$ and hence in this case from~\eqref{eq:Large d} 
and Theorem~\ref{thm:Sum 2} we have
$$
S_{n}(\psi; a,b) = \sum_{P \in \E(\F_q)} 
 \psi\(ax(P) +b x(nP)\)=\sum_{P \in \E(\F_q)} 
 \psi\(ax(P)\)\ll \sqrt{q}.
$$

\subsection{Sums with the elliptic curve power generator}

We now improve the  results of~\cite{BFGS,ElMaSh} on the distribution of  the {\it power
generator\/} on elliptic curves. Namely, given a point $G \in 
\E(\F_q)$ of order
$t$, we fix an integer $e$ with $\gcd(e,t) =1$, put $W_0 = G$ and consider the
sequence 
\begin{equation} \label{ECPow} 
W_n =  eW_{n-1}, \qquad n = 
1, 2, \ldots.
\end{equation}
In a more explicit form we have $W_n = e^n G$. Clearly, the sequence $W_n$ is periodic with period $T$ which is the multiplicative order of $e$ modulo $t$.

For a point $G\in\E(F_q)$, a nonprincipal additive character $\psi$ of $\F_q$
and an integer $N$, we consider character sums 
$$
S(G,\psi,N)) = \sum_{n=0}^{N-1} \psi \(x(W_n)\) 
$$ 
with the sequence~\eqref{ECPow}.

For $N = T$ the sum $S(G,\psi,T)$ is estimated in~\cite{LanShp1},
where it is shown that for any fixed positive integer $\nu$,
we have
$$
S(G,\psi,T) \ll
T ^{1 - (3\nu + 2)/2\nu(\nu + 2)}
t^{(\nu + 1)/\nu (\nu+2)}  q^{1/4(\nu + 2)}.
$$

In~\cite{ElMaSh}, using two different approaches the above result is extended to incomplete sums $S(G,\psi,N)$ with $N\le T$. 
One of the approaches has led  to
\begin{equation}
\label{Old-short-bound}
 S (G,\psi,N) 
  \ll N^{1 - (3\nu + 2)/2\nu(\nu + 3)}
t^{(\nu+1)/\nu(\nu+3)}q^{1/4(\nu+3)}, 
\end{equation}
while the other one has yielded  
\begin{equation}
\label{Old-long-bound}
  S (G,\psi,N)   \ll
T ^{1 - (3\nu + 2)/2\nu(\nu + 2)}
t^{(\nu + 1)/\nu (\nu+2)}  q^{1/4(\nu + 2)} \log q.
\end{equation}

Notice that the  bound~\eqref{Old-short-bound}  is stronger than~\eqref{Old-long-bound} 
for  short sums but for almost complete sums, the bound~\eqref{Old-long-bound}  is 
stronger.

Here using Lemma~\ref{lem:BillSum-M} and an inductive argument, we give a bound 
that improves both~\eqref{Old-short-bound} and~\eqref{Old-long-bound}. 

\begin{theorem}\label{Short-powergenerator}
\label{Short} 
Let $\E$ be an ordinary elliptic curve defined over $\Fq$ and let $N \le T$.
Suppose that for some fixed
$\varepsilon > 0$ we have $t \ge q^{1/2 + \varepsilon}$.
Then for any fixed integer $\nu \ge 1$ there exists $C(\nu,\varepsilon)\ge 1$ depending only on $\nu$ and $\varepsilon$ such that 
$$
 S (G,\psi,N)  
  \le C(\nu,\varepsilon) N^{1 - (3\nu + 2)/2\nu(\nu + 2)}
t^{(\nu+1)/\nu(\nu+2)}q^{1/4(\nu+2)}(\log q)^{1/(\nu+2)}.
$$
\end{theorem}

\begin{proof}
Our proof is based on an induction. 

Notice that if $N\le q^{1/2}$, then since $t \ge q^{1/2 + \varepsilon}$ we have
$$
N^{1 - (3\nu + 2)/2\nu(\nu + 2)}
t^{(\nu+1)/\nu(\nu+2)}q^{1/4(\nu+2)}(\log q)^{1/(\nu+2)}\ge N,
$$
and thus the claim holds trivially.

Now suppose that the claim is true for all $k<N$, and hence there exists $C(\nu, \varepsilon)$, which is
to be determined later, depending only on $\nu$ and $\varepsilon$, so that for all $k< N$, we have
$$
 S (G,\psi,k)  
  \le C(\nu,\varepsilon) k^{1 - (3\nu + 2)/2\nu(\nu + 2)}
t^{(\nu+1)/\nu(\nu+2)}q^{1/4(\nu+2)}(\log q)^{1/(\nu+2)}.
$$

Let $\cM = \{0, \ldots, M-1\}$ where $M<N$.
For every $m\in \cM$, we have 
$$
S(G,\psi,N)= \sum_{n=0}^{N-1} \psi \(x(e^{n+m} G)\)+S(G,\psi,m)-
S(H,\psi,m), 
$$
where $H=e^NG$, and hence 
$$
\sum_{m=0}^{M-1}S(G,\psi,N)= \sum_{m=0}^{M-1}\sum_{n=0}^{N-1} \psi \(x(e^{n+m} G)\)+\sum_{m=0}^{M-1}S(G,\psi,m)-
\sum_{m=0}^{M-1}S(H,\psi,m). 
$$
Notice that our bounds hold for any point of order $t$, and thus
using the fact that $\gcd(e,t) =1$ we can apply the induction hypothesis 
to the point $H$ too. Hence by the induction hypothesis we have
\begin{equation*}
\begin{split}
M |S&(G_0,\psi,N)|\\
&\le W+ 2MC(\nu,\varepsilon)M^{1 - (3\nu + 2)/2\nu(\nu + 2)}
t^{(\nu+1)/\nu(\nu+2)}q^{1/4(\nu+2)}(\log q)^{1/(\nu+2)},
\end{split}
\end{equation*}
where
$$
W =\left|\sum_{m=0}^{M-1} \sum_{n=0}^{N-1} \psi \(x(e^{n+m} G)\)\right|.  
$$
Applying Lemma~\ref{lem:BillSum-M}, we get
$$
W\le D(\nu,\varepsilon) (M)^{1-1/2\nu} (N)^{1-1/(\nu+2)}t^{(\nu+1)/\nu(\nu+2)}q^{1/4(\nu+2)}(\log q)^{1/(\nu+2)}
$$
for some $D(\nu,\varepsilon)$ depending only on $\nu$ and $\varepsilon$.
From the two inequalities above, we have
\begin{equation*}
\begin{split}
|S&(G_0,\psi,N)|\\
& \le D(\nu,\varepsilon) M^{-1/2\nu}N^{1-1/(\nu+2)}t^{(\nu+1)/\nu(\nu+2)}q^{1/4(\nu+2)}(\log q)^{1/(\nu+2)}\\
&\qquad +2C(\nu,\varepsilon)M^{1 - (3\nu + 2)/2\nu(\nu + 2)}
t^{(\nu+1)/\nu(\nu+2)}q^{1/4(\nu+2)}(\log q)^{1/(\nu+2)}.
\end{split}
\end{equation*}

We see that it suffices to take $M=\lceil N/2\rceil$ and
$$
C(\nu,\varepsilon)=\frac{2^{1/2\nu}}{1-2^{(3\nu+2)/2\nu(\nu+2)}}D(\nu,\varepsilon)
$$
to conclude the proof. \end{proof}

Notice that when $t=q^{1+o(1)}$ which is the most interesting case, taking $\nu$ to be a very large number shows that the
bound in Theorem~\ref{Short-powergenerator} is stronger than the bound~\eqref{Old-short-bound} whenever
$N\ge q^{{5/6}+\varepsilon}$ for some fixed $\varepsilon > 0$.

\section{Comments}

Dvir~\cite{Dvir} has considered the problem of constructing randomness extractors 
for algebraic varieties. In general terms the problem can be described as follows.
Given an algebraic variety $\cV$ over $\F_q$ and one or several sources 
of random but not necessarily uniformly generated points on $\cV$, design 
an algorithm to generate  long strings  of random bits with a distribution 
that is close to uniform. 
The construction of~\cite{Dvir} requires only one but rather uniform source of 
points on $\cV$. In the case when $\cV = \cE$, the result of Theorem~\ref{thm:BillSum-A}
has a natural interpretation as a two-source extractor from two biased 
sources of points $P$ and $Q$, respectively. Say, if $q=p$, then one can use most significant bits
of $x(P\oplus Q)$ (in some standard representation of the residues modulo $p$).
The exact number of output bits depends on the bias of the sources of points 
$P$ and $Q$. 

We also remark that many of our results have direct analogues for sums with
multiplicative characters.

\section{Acknowledgements}

During the preparation of this paper,
O.A. was  supported in part by a grant from IPM  Grant~91050418 (Iran)
and I.~S. by ARC Grant~DP130100237 (Australia) and by NRF Grant~CRP2-2007-03 (Singapore). 

A portion of this work was done when the authors were visiting the University of Waterloo;
the support and hospitality of this institution are gratefully acknowledged.

\end{document}